\documentclass[12pt,twoside]{amsart}
\usepackage{amssymb,amsmath,amsthm, amscd, enumerate, mathrsfs, upgreek}
\usepackage{graphicx, hhline, tikz}
\usepackage[colorlinks=true,pagebackref,hyperindex]{hyperref}
\usepackage[all]{xy}
\usepackage{color}
\usepackage[backrefs, alphabetic, initials]{amsrefs}
\usepackage{color} 
\usepackage{latexsym}
\usepackage[T1]{fontenc} 
\usepackage{fancyhdr}
\usepackage{enumerate}
\usepackage{tikz-cd}
\usepackage{array, longtable}
\usepackage{caption}
\usepackage{mathtools}
\newcolumntype{C}{>{$}c<{$}}
\newcolumntype{L}{>{$}l<{$}}

\title[Fano index]
{On Fano indices of weighted projective spaces}

\date{\today, version 0.01}

\author{Haidong Liu}
\address{Sun Yat-sen University, School of Mathematics, Guangzhou, 510275, China}
\email{liuhd35@mail.sysu.edu.cn, jiuguiaqi@gmail.com}

\DeclareMathOperator{\Cl}{Cl}

\DeclareMathOperator{\loc}{loc}
\DeclareMathOperator{\ord}{ord}
\DeclareMathOperator{\im}{im}
\DeclareMathOperator{\wps}{wps}
\DeclareMathOperator{\term}{term}
\DeclareMathOperator{\vol}{vol}
\DeclareMathOperator{\conv}{conv}
\DeclareMathOperator{\age}{age}
\DeclareMathOperator{\PS}{PS}

\newtheorem{thm}{Theorem}[section]
\newtheorem{lem}[thm]{Lemma}
\newtheorem{prop}[thm]{Proposition}
\newtheorem{conj}[thm]{Conjecture}
\newtheorem{cor}[thm]{Corollary}
\newtheorem{ques}[thm]{Question}

\theoremstyle{definition}

\newtheorem{defn}[thm]{Definition}
\newtheorem{rem}[thm]{Remark}
\newtheorem*{ack}{Acknowledgments}

\newtheorem{case}{Case}
\newtheorem*{claim}{Claim}

\newtheorem*{ai}{AI disclosure} 
\newtheorem*{interest}{Interest disclosure} 

\makeatletter
 
 \@addtoreset{equation}{section}
\makeatother

\begin{document}

\begin{abstract}
    The Sylvester sequence is defined recursively by $s_1=2$ and $s_i=s_{1}\cdots s_{i-1}+1$. In this paper, we prove that the Fano index of an $n$-dimensional well-formed weighted projective space with canonical singularities is bounded above by $(s_n-1)(2s_n-3)$. This gives an affirmative answer to a conjecture of Chengxi Wang for weighted projective spaces and $\mathbb Q$-factorial toric Fano varieties with Picard number one. 
    
    We also investigate the distribution of Fano indices among $4$-dimensional weighted projective spaces. As the distribution of Fano indices of weighted projective spaces coincides with that of indices of terminal Calabi--Yau varieties in dimension $n\leq 3$, we expect this coincidence to persist also in dimension 4, and more generally, in all dimensions.
\end{abstract}

\maketitle 
\tableofcontents

\section{Introduction}\label{sec1}

A normal projective variety is called \emph{Fano} if its anti-canonical divisor is ample. Weighted projective spaces provide a standard class of examples. For a Fano variety $X$ with canonical singularities, its \emph{Fano index} is defined to be
\[
    q(X) \coloneq \max\{q\in\mathbb Z_{>0} \mid -K_X \sim_{\mathbb Q}qA, \quad A\in \Cl (X) \}. 
\]
This invariant has attracted increasing attention over the last few decades \cites{bs,prokhorov2010, prokhorov2013,kasprzyk, liu-liu1, liu-liu2, jiang-liu-liu, jiang-liu}. In the case of a well-formed weighted projective space $X=\mathbb P(a_1,...,a_{n+1})$, the Fano index is simply the sum of the weights: 
\begin{equation}\label{eq.indexofwps}
    q(X)= a_1+\cdots +a_{n+1}.
\end{equation}

Evidence from numerous examples in the study of Fano varieties suggests that some extremal cases---e.g., those attaining maximal degree or maximal Fano index---are precisely weighted projective spaces. Moreover, in these extremal cases, the associated weights are closely connected with the \emph{Sylvester sequence}.

The Sylvester sequence is defined recursively by $s_1=2$ and $s_i=s_{1}\cdots s_{i-1}+1$, with initial terms $2,3,7,43,1807$; its terms are pairwise coprime.  As $s_i-1$ appears frequently in this paper, we set $y_i\coloneq s_i-1$, and recall their relations as follows:
\begin{align}\label{eq.sylseq}
    y_i=s_i-1&=y_{i-1}s_{i-1}=s_{i-1}(s_{i-1}-1)=y_{i-1}(y_{i-1}+1),\\
    1-\frac{1}{y_i}&=1-\frac{1}{s_i-1}=\frac{1}{s_1}+\cdots +\frac{1}{s_{i-1}}.
\end{align}

Beginning with Hensley’s boundedness theorem \cite{hensley}, Averkov--Kr\"{u}mpelmann--Nill studied in \cite{akn}*{Theorem 2.2} the integral simplices with a unique interior integral point. Their results derived upper bounds for the degrees and Fano indices of $\mathbb Q$-factorial toric Fano varieties with Picard number one and canonical singularities, both of which are $2y_n^2$ (see \eqref{eq.knownb} and the proof of Corollary \ref{cor.pic1} for the Fano index bound). Although the degree bound is sharp, the Fano index bound is not; in more general setting \cite{wang}, Wang conjectured that the sharp Fano index bound should be slightly smaller than $2y_n^2$:

\begin{conj}[\cite{wang}*{Conjecture 3.7}]\label{conj.main}
    Let $X$ be an $n$-dimensional Fano variety with canonical singularities. Then its Fano index is at most $y_n(2y_n-1)$.
\end{conj}

In \cite{wang}*{Theorem 3.2}, Wang showed that the upper bound $q=y_n(2y_n-1)$ is achievable, attained by the weighted projective space $\mathbb P (\frac{q}{s_1},\frac{q}{s_2}, \dots,\frac{q}{s_{n-2}}, y_{n-1}, y_{n-1}-1)$. Recently, Jiang and the author \cite{jiang-liu} verified Conjecture \ref{conj.main} in dimension 3. In this paper, we confirm this conjecture for weighted projective spaces and  $\mathbb Q$-factorial toric Fano varieties with Picard number one.

\begin{thm}\label{thm.main}
    The Fano index of an $n$-dimensional well-formed weighted projective space with canonical singularities is at most $y_n(2y_n-1)$.
\end{thm}

A $\mathbb Q$-factorial toric Fano variety $X$ with Picard number one is also called a \emph{fake weighted projective space}, which is a building block of the toric minimal model program (cf. \cite{fujino2003}*{\S 1}). It is known that such a variety admits a finite Galois \'etale in codimension one morphism $\pi\colon Y\coloneq \mathbb P(a_1,\dots, a_{n+1})\to X$ from a well-formed weighted projective space $Y$ (see \cite{kasprzyk2009}*{Proposition 2.2}); moreover, $Y$ has canonical singularities if $X$ does \cite{kollar-mori}*{Proposition 5.20}. Also, if $-K_X\sim_{\mathbb Q}qA$, then $\pi^*A$ is a Weil divisor on $Y$ and
\[
    -K_Y=\pi^*(-K_X)\sim_{\mathbb Q}q(\pi^*A),
\]
which implies that $q(X)\leq q(Y)$. Therefore, as a direct consequence of Theorem \ref{thm.main}, we obtain the following:

\begin{cor}\label{cor.pic1}
    The Fano index of an $n$-dimensional $\mathbb Q$-factorial toric Fano variety with Picard number one and canonical singularities is at most $y_n(2y_n-1)$.
\end{cor}

We are interested not only in the bounds of Fano indices of weighted projective spaces, but also in their distribution. In dimension 4, we find that the Fano indices take values in a set associated with the Euler function $\varphi$.

\begin{thm}[Theorem \ref{thm.disindim4}]
    The Fano index of a $4$-dimensional well-formed weighted projective space with canonical singularities lies in the set $\{m\in\mathbb Z_{>0}\mid \varphi(m)\leq 984\}$.
\end{thm}

In fact, we obtain an explicit list of Fano indices of weighted projective spaces up to dimension 4. As observed in \cite{liu2025}*{\S 1.1}, for dimension $n\leq 3$, these distributions  coincide with the indices of $n$-dimensional terminal Calabi--Yau varieties and the indices of an $n$-dimensional log canonical singularity $P\in X$. We conjecture that this coincidence persists in all dimensions.

\begin{conj}[Conjecture \ref{conj.coincide}]
    The set consisting of the Fano indices of $n$-dimensional well-formed weighted projective spaces with canonical singularities together with $\{1,\dots,n\}$ coincides with both the set of indices of $n$-dimensional terminal Calabi--Yau varieties and the set of indices of an $n$-dimensional log canonical singularity $P\in X$.
\end{conj}

There are some evidences supporting this conjecture, see Proposition \ref{prop.leq3} and  \cites{etw, wang, masamura} for examples. During the preparation of this paper, Jie Liu kindly pointed out to the author that Jihao Liu has provided on his website a counterexample to \cite{etw}*{Conjecture 7.10} in dimension 159. While the construction in \cite{etw}*{Conjecture 7.10} fails, the explicit upper bound might still be $y_n(2y_n-1)$.  

\quad

\textbf{Proof strategy of Theorem \ref{thm.main}.} The proof combines arithmetic information supplied by the age function with product estimates for the barycentric coordinates of the associated lattice simplex. Assuming that the Fano index exceeds the conjectural bound, we first determine the two smallest weights. A majorization argument then provides sufficiently sharp upper bounds for the remaining weights. These bounds force the age at $2y_n-1$ to be $n$, contradicting canonicality.

\begin{ack} 
The author would like to thank Chen Jiang, Alexander Kasprzyk, Jie Liu, and Benjamin Nill for some useful discussions. He would also like to thank Akiyoshi Sannai for pointing out a gap in Corollary 2.5 (3) of previous preprint. The author is supported in part by the National Key Research and Development Program of China (No. 2023YFA1009801) and NSFC (No. 12571048).  
\end{ack}

\begin{ai}
The author used ChatGPT and Deepseek as research-assistance tools during the preparation of this manuscript. In particular, ChatGPT Pro 5.5 assisted in exploratory discussions, brought the majorization theory and Karamata’s inequality to the author’s attention, and helped together with Deepseek translate and adapt computational code from Python to multithreaded C++. They were also used for limited language editing. 
The paper was written by the author, who takes full responsibility for the accuracy and content of the paper.
\end{ai}

\begin{interest}
    The authors report there are no competing interests to declare.
\end{interest}

\section{Preliminaries}

Throughout this paper, we follow \cite{cls} for the general framework of toric geometry and \cite{kollar-mori} for the theory of singularities.

\subsection{Weighted projective space}

Let $a_1\geq \dots \geq  a_{n+1}>0$ be positive integers. A \emph{weighted projective space} $X=\mathbb P(a_1,...,a_{n+1})$ is the quotient variety $(\mathbb A^{n+1}\backslash 0)/\mathbb G_m$, where the multiplicative group $\mathbb G_m$ acts by $t\cdot(x_1,\dots,x_{n+1})=(t^{a_1}x_1, \dots, t^{a_{n+1}}x_{n+1})$. We say that $X$ is \emph{well-formed} if $\gcd(a_1,\dots, \hat a_i,\dots, a_{n+1})=1$ for every $i$.  For details, see \cites{dolgachev,if} for example. 

We restrict our attention to well-formed weighted projective spaces in this paper. The canonical divisor of a well-formed weighted projective space $X=\mathbb P(a_1,...,a_{n+1})$ is given by
\[
    \mathcal O_X(K_X)\simeq\mathcal O_X(-a_1-\cdots-a_{n+1}),
\]
where $\mathcal O_X(1)$ is the ample generator of the Weil divisor class group. Consequently, the Fano index of $X$ is
\[
    q(X)= a_1+\cdots +a_{n+1}.
\]
\subsection{Age function}
The age function, which goes back to Miles Reid’s study of cyclic quotient singularities \cite{reid}, provides an effective way to encode the singularities of a weighted projective space in terms of its weights. In particular, terminality/canonicality can be translated into explicit arithmetic restrictions on the values of the age function. We will use repeatedly Kasprzyk's criterion \cite{kasprzyk}, somewhat different from Reid's original one, to derive bounds on the weights and to exclude possible Fano indices.

Let $X=\mathbb P(a_1,...,a_{n+1})$ be a well-formed weighted projective space of dimension $n\geq 4$, where $a_1\geq \cdots \geq a_{n+1}>0$. Let $q=a_1+\cdots +a_{n+1}$ be the Fano index of $X$.

\begin{defn}
    The \emph{age function} of $X$ is defined as
    \[
        \age(k)\coloneq \sum_{i=1}^{n+1}\{\frac{a_ik}{q}\}=k-\sum_{i=1}^{n+1}\lfloor\frac{a_ik}{q}\rfloor\in \mathbb Z_{\geq 0}
    \]
    for an integer $k$.
\end{defn}

The following results are easy properties about the age function.

\begin{prop}\label{prop.age>0}
    $\age(k)>0$ for any integer $0< k< q$. 
\end{prop}

\begin{proof}
    If $\age(k)=0$, then $\{\frac{a_ik}{q}\}=0$, that is, $q\mid a_ik$ for any $i$. It follows that $\frac{q}{\gcd(q,k)}\mid a_i$ for any $i$, and hence $\frac{q}{\gcd(q,k)}\mid \gcd(a_1,\dots,a_{n+1})=1$. Then $q\mid k$, which is impossible.
\end{proof}

\begin{prop}\label{prop.symmetry}
    $\age(k)+\age(q-k)=n+1-\#\{i\mid q \mid a_ik\}$. 
\end{prop}

\begin{proof}
    We have $\age(k)+\age(q-k)=\sum_{i=1}^{n+1}(\{\frac{a_ik}{q}\}+\{\frac{a_i(q-k)}{q}\})=\sum_{i=1}^{n+1}(\{\frac{a_ik}{q}\}+\{\frac{-a_ik}{q}\})$. As $\{x\}+\{-x\}=0$ (resp. $=1$) if $x$ is (resp. not) an integer, the conclusion follows.
\end{proof}

Therefore, $1\leq \age(k)\leq n$ by Propositions \ref{prop.age>0} and \ref{prop.symmetry} for $0<k<q$. The following age criterion is given by Kasprzyk:

\begin{prop}[\cite{kasp}*{Proposition 2.5}]\label{prop.criterion}
    $X$ has at worst canonical singularities if and only if $\age(k)\neq n$ for any $2\leq k\leq q-2$.
\end{prop}

The following corollary is a warm-up for our use of the age criterion.

\begin{cor}\label{cor.bounds}
    Suppose $X=\mathbb P(a_1,...,a_n, a_{n+1})$ has canonical singularities and $q\geq 2n+1\geq 9$. Then
    \begin{enumerate}
        \item (\cite{kasprzyk2009}*{Theorem 3.5}) $a_i\leq \frac{q}{1+i}$ for any $i$.
        \item $a_1\geq \frac{q}{n}$.
        \item For $1< i\leq n$, if $a_1<\frac{2q}{n-1+i}$, then $a_i\geq \frac{q}{n-1+i}$.
    \end{enumerate}
\end{cor}

\begin{proof}
    For (1), suppose in contrast that $a_1\geq \cdots \geq a_i>\frac{q}{1+i}$. If $a_1\geq \frac{2q}{1+i}$, then $a_1+\cdots +a_i> \frac{2q}{1+i}+\frac{(i-1)q}{1+i}=q$, which is impossible (the case $i=1$ is obvious). So we have 
    \[
        \frac{2q}{1+i}>a_1\geq \cdots \geq a_i>\frac{q}{1+i}.
    \]
    Set $k\coloneq 1+i$. On the one hand, we have $2q>a_jk>q$ for any $j\leq i$, so $q\nmid a_jk$ for $j\leq i$. On the other hand, for $i\leq n$, we have 
    \[
    a_{i+1}\leq q-a_i-\cdots-a_1<q-\frac{iq}{1+i}=\frac{q}{1+i},
    \]
    and hence $0<ka_{n+1}\leq\cdots \leq  ka_{i+1}<q$ for $i\leq n$. 

    Then $\age(1+i)=\age(k)=\sum_{j=1}^{i}\{\frac{a_jk}{q}\}+\sum_{j= i+1}^{n+1}\{\frac{a_jk}{q}\}=\sum_{j=1}^{i}(\frac{a_jk}{q}-1)+\sum_{j=i+1}^{n+1}\frac{a_jk}{q}=k-i=1$. Moreover, we have $\#\{j\mid q \mid a_jk\}=0$ for any $j$ by above discussions. By Proposition \ref{prop.symmetry}, we have $\age(q-k)=n$, contradicting Proposition \ref{prop.criterion}.

    For (2), suppose in contrast that $\frac{q}{n}>a_1\geq \cdots \geq a_{n+1}>0$. Then $\age(n)=\sum_{j=1}^{n+1}\{\frac{a_jn}{q}\}=\sum_{j=1}^{n+1} \frac{a_jn}{q} =n$, contradicting Proposition \ref{prop.criterion}. 

    For (3), we prove it inductively. At the beginning case $i=2$, we have  $q<\frac{n+1}{n}q\leq (n+1)a_1<2q$ by (2) and assumption, hence $\{\frac{a_1(n+1)}{q}\}=\frac{a_1(n+1)}{q}-1$. Suppose in contrast that $a_2<\frac{q}{n+1}$. Then 
    \[
        \age(n+1)=\sum_{j=1}^{n+1}\{\frac{a_j(n+1)}{q}\}=\sum_{j=1}^{n+1}\frac{a_j(n+1)}{q}-1=n,
    \]
    contradicting Proposition \ref{prop.criterion}. Therefore, we must have $a_2 \geq \frac{q}{n+1}$.  

    Suppose for $j<i$, we have the conclusions. For $i$, we have $a_1<\frac{2q}{n-1+i}<\frac{2q}{n-1+j}$ for any $1\leq j<i$. So by the inductive assumption, we have $a_j\geq \frac{q}{n-1+j}$ for $1\leq j<i$.  Hence
    \[
        1<\frac{n-1+i}{n-1+j}\leq \frac{a_j(n-1+i)}{q}\leq \frac{a_1(n-1+i)}{q}<2
    \]
    for any $1\leq j<i$.
    If $\frac{q}{n-1+i}>a_i\geq \cdots \geq a_{n+1}$, then 
    \[
        \age(n-1+i)=\sum_{j=1}^{n+1}\{\frac{a_j(n-1+i)}{q}\}=\sum_{j=1}^{n+1}\frac{a_j(n-1+i)}{q}-\sum_{j=1}^{i-1} 1=n-1+i-(i-1)=n,
    \]
    contradicting Proposition \ref{prop.criterion} as $n-1+i\leq q-2$. Therefore, we must have $a_i\geq \frac{q}{n-1+i}$.
\end{proof}

\subsection{Integral simplex with one interior integral point}

Let $X=\mathbb P(a_1,...,a_{n+1})$ be a well-formed weighted projective space of dimension $n\geq 4$ with canonical singularities, where $a_1\geq \cdots \geq a_{n+1}>0$. Let $q=a_1+\cdots +a_{n+1}$ be the Fano index of $X$, and set $\beta_i\coloneq \frac{a_i}{q}$ for each $i$. Then $\beta_1\geq \cdots \geq \beta_{n+1}>0$ and $\sum_{i=1}^{n+1}\beta_i=1$. 

It is known (see, e.g., \cite{bb}*{Proposition 2}, \cite{conrads} and  \cite{kasp}*{introduction}) 
that the integral simplex 
\[
    S\coloneq \conv\{v_1,\dots, v_{n+1}\}
\]
\emph{associated with} $X$ has a unique interior point $0$ such that
\begin{enumerate}
    \item $\beta_1v_1+\cdots +\beta_{n+1}v_{n+1}=0$;
    \item the primitive vertices $v_i$ generate the full lattice. 
\end{enumerate}
Here $\beta_1,\dots \beta_{n+1}$ are called the \emph{barycentric coordinates} of the interior point $0$. Denote by $\vol(S)$ the \emph{normalized volume} of $S$, i.e., $n!$ times its Euclidean volume. Let $F_i\coloneq \conv \{v_1,\dots, \hat{v_i},\dots, v_{n+1}\}$ be the facet of $S$ opposite the vertex $v_i$ and $S_i\coloneq \conv(\{0\}\cup F_i)$. 
By considering the signed maximal minors of the matrix with columns $v_1,\dots,v_{n+1}$, we have $\vol(S_i)=|\det F_i|=a_i$. Summing the volumes of the star simplices $S_i$ yields that
\begin{equation}\label{eq.ind-vol}
     \vol(S)=\sum_{i=1}^{n+1} a_i=q.
\end{equation}
In particular, \cite{akn}*{Theorem 2.2} gives
\begin{equation}\label{eq.knownb}
    q=\vol(S)\leq 2y_n^2.
\end{equation}

\begin{rem}
    Note that the volume $\vol(S)$ of the associated simplex $S$ defined in this paper differs from that in \cite{nill}*{Theorem C} and \cite{akn}*{Theorem 2.5} by a constant factor. Under our definition, $\vol(S)$ cannot attain the upper bound $2y_n^2$. For instance, the weighted projective space $\mathbb P(\frac{2y_n}{s_1},\dots,\frac{2y_n}{s_{n-1}}, 1,1)$ (also its corresponding simplex defined in \cites{nill,akn}) attaining the maximal degree $2y_n^2$ (see \cite{akn}*{Theorem 2.11}) has Fano index $q=\frac{2y_n}{s_1}+\cdots+\frac{2y_n}{s_{n-1}}+1+1=2y_n$, which differs from $2y_n^2$ by a factor $y_n$.
\end{rem}


In the following, we summarize several key results concerning bounds of barycentric coordinates of the interior point $0$.

\begin{thm}[product-sum inequalities, \cite{averkov}*{Theorem 1.1}]\label{thm.psineq}
    Keeping the notation established at the beginning of this subsection, we have
    \begin{equation}\label{eq.psineq}
        \beta_1\cdots \beta_j\leq \beta_{j+1}+\cdots +\beta_{n+1}
    \end{equation}
    for any $1\leq j\leq n$.
\end{thm}

\begin{thm}[\cite{pikhurko}*{Lemma 5}, \cite{averkov}*{Theorem 3.7}, \cite{akn}*{Theorem 3.7}]\label{thm.ubton-1}
    Keeping the notation established at the beginning of this subsection, we have
    \[
        \beta_1\cdots \beta_{n}\leq \frac{1}{\vol(S)}=\frac{1}{q}.
    \]
\end{thm}

\begin{thm}[\cite{akn}*{Theorem 2.1}]\label{thm.lbofn-1}
    Keeping the notation established at the beginning of this subsection, we have
    \[
        \beta_{n}\geq \frac{1}{2y_n}.
    \]
\end{thm}

The following lemma is a special case of \cite{akn}*{Theorem 4.3}, from which we derive the following more concrete results.

\begin{lem}\label{lem.lbupton-2} 
    Keeping the notation established at the beginning of this subsection, we have
    \[
        \beta_1\cdots \beta_{n-1}\geq \frac{9}{32y_{n-1}^2}>\frac{1}{4y_{n-1}^2}.
    \]
\end{lem}

\begin{proof}
    By \cite{akn}*{Lemma 4.2(a)}, there exists an Izhboldin--Kurliandchik minimizing solution of function $x_1\cdots x_{n-1}$ attained by
    \[
        y(l)=\left(\frac{1}{s_1},\cdots, \frac{1}{s_{l-1}}, \frac{1}{(n-l+2)(s_l-1)}, \cdots,  \frac{1}{(n-l+2)(s_l-1)}\right)
    \]
    for some $1\leq l\leq n+1$. Note that $y(1)=(\frac{1}{n+1},\cdots, \frac{1}{n+1})$. By  \cite{akn}*{Claim 4.3.1}, we can further assume that $l\leq n-1$. Therefore,
    we have
    \[
        \beta_1\cdots \beta_{n-1}\geq \frac{1}{s_1\cdots s_{l-1}(n-l+2)^{n-l}(s_l-1)^{n-l}}=\frac{1}{(n-l+2)^{n-l}y_l^{n-l+1}}
    \]
    for some $1\leq l\leq n-1$. In the case $n=4$, simple calculation shows that 
    \[
        \beta_1\beta_2 \beta_{3}\geq \frac{1}{128}=\frac{9}{32y_3^2}>\frac{1}{144}=\frac{1}{4y_3^2}.
    \]
    In the case $n\geq 5$,  we show that the function $g(l)=(n-l+2)^{n-l}y_l^{n-l+1}$ is increasing in $l$. Indeed, if $l\geq 4$, then $s_l\geq s_4=43>4e$, where $e$ is the Euler's number, and hence \cite{akn}*{Claim 4.3.3} provides that $g(l)\leq g(l+1)$; if $l\leq 3$, we can check by $(1+\frac{1}{m})^m<e$ that 
    \[
        g(1)=(n+1)^{n-1}\leq g(2)=2^{n-1}n^{n-2}\leq g(3)=6^{n-2}(n-1)^{n-3}\leq g(4)=42^{n-3}(n-2)^{n-4}.
    \]
    In conclusion, $g(l)$ is increasing, and hence the maximum occurs at $l=n-1$. Therefore, 
    \[
        \beta_1\cdots \beta_{n-1}\geq \frac{1}{3y_{n-1}^2}> \frac{9}{32y_{n-1}^2}>\frac{1}{4y_{n-1}^2}. \qedhere
    \]
\end{proof}

Generalizing the product-sum inequalities \eqref{eq.psineq}, we obtain the following \emph{generalized product-sum inequalities}  \eqref{eq.tailpsineq}, which will be referred to as $\PS(j)$. The proof of the following lemma is analogous to those in \cite{akn}*{Lemma 4.2}. We will give an alternative proof in subsection \ref{sub.majority}. Both proofs rely on \cite{akn}*{Lemma 4.5}.

\begin{lem}\label{lem.keygap}
    Let $x_1\geq x_2\geq \cdots \geq x_m>0$, where $x_1+\cdots +x_m=1-t$ such that $0<t<\frac{1}{y_m}$. Suppose that 
    \begin{equation}\label{eq.tailpsineq}
        x_1\cdots x_j\leq x_{j+1}+\cdots +x_m+t
    \end{equation} 
    for all $1\leq j\leq m-1$. Then 
    \[
        x_1\cdots x_m\geq \frac{1-ty_m}{y_m^2}.
    \]
\end{lem}

\begin{proof}
    Fix an integer $k$ such that $1\leq k\leq m-1$. As 
    \[
        x_1\cdots x_j\leq x_{j+1}+\cdots +x_m+t=1-\sum_{i=1}^j x_i
    \]
    for any $1\leq j\leq k$, we may apply \cite{akn}*{Lemma 4.5} to $x_1,\dots, x_k$ and obtain that 
    \[
        x_1+\cdots +x_k\leq \frac{1}{s_1}+\cdots +\frac{1}{s_k}.
    \]
    In particular,
    \[
        x_1+\cdots +x_{m-1}\leq \frac{1}{s_1}+\cdots +\frac{1}{s_{m-1}}=1-\frac{1}{y_m}.
    \]
    It follows that $x_m=1-t-(x_1+\cdots +x_{m-1})\geq \frac{1}{y_m}-t>0$. Therefore, the domain
    \[
        \mathcal X^m_t \coloneq 
        \left\{
        (x_1,\ldots,x_m)\in\mathbb R^m \;\middle|\;
        \begin{aligned}
        &x_1+\cdots +x_m=1-t,\\
        &x_1\geq \cdots \geq x_m\geq \frac{1}{y_m}-t,\\
        &x_1\cdots x_j \leq x_{j+1}+\cdots+x_m+t \quad (1\le j\le m-1)
        \end{aligned}
        \right\}
    \]
    is compact. So we can consider the Izhboldin–Kurliandchik minimizing solution of function $x_1\cdots x_m$ on the domain $\mathcal X_t^m$. We claim that at a minimizing solution of $x_1\cdots x_m$, every inequality $\PS(j)$ in \eqref{eq.tailpsineq} must be an equality.

    Suppose in contrast that $\PS(j')$ is strict for some $j'$. Then we can find $1\leq j_1\leq j'\leq j_2\leq m-1$ such that $\PS(j)$ is strict for any $j_1\leq j\leq j_2$, and $\PS(j)$ are equalities for $j=j_1-1$ and $j=j_2+1$. Here we formally set $\PS(0)$ and $\PS(m)$ as equality. As $x_1\cdots x_{j_1-1}=x_{j_1}+\cdots +x_m+t>x_{j_1}$, we have $x_{j_1-1}>x_1\cdots x_{j_1-1}>x_{j_1}$; similarly we have $x_{j_2+1}>x_{j_2+2}$. Hence, for a sufficiently small $\varepsilon>0$, replacing $x_{j_1}$ by $x_{j_1}+\varepsilon$ and $x_{j_2+1}$ by $x_{j_2+1}-\varepsilon$, the sum and the order in definition of $\mathcal X_t^m$ remain valid. For the inequality $\PS(j)$, if $j<j_1$, then the modification lies on the right hand side of $\PS(j)$, hence nothing is changed; if $j_1\leq j\leq j_2$, then the inequality $\PS(j)$ is strict, hence remains valid when $\varepsilon$ is sufficiently small; if $j> j_2$, then the modification lies on the left hand side of $\PS(j)$, where the product is multiplied by 
    \[
        \frac{(x_{j_1}+\varepsilon)(x_{j_2+1}-\varepsilon)}{x_{j_1}x_{j_2+1}}=1+\frac{\varepsilon(x_{j_2+1}-x_{j_1})-\varepsilon^2}{x_{j_1}x_{j_2+1}}<1.
    \]
    Therefore, the perturbation is feasible, while the total product is strictly smaller. This contradicts minimality,  and hence every $\PS(j)$ is an equality.

    Then the first equality  $\PS(1)$ gives $x^*_1=1-x^*_1$, i.e., $x^*_1=\frac{1}{s_1}$. Inductively, we have $x^*_j=\frac{1}{s_j}$ for $1\leq j\leq m-1$. It follows that $x^*_m=1-t-(\frac{1}{2}+\cdots +\frac{1}{s_{m-1}})=\frac{1}{y_m}-t$ and 
    \[
        x_1\cdots x_m\geq x^*_1\cdots x^*_m=\frac{1}{s_1}\frac{1}{s_2}\cdots\frac{1}{s_{m-1}}(\frac{1}{y_m}-t)=\frac{1-ty_m}{y_m^2}. \qedhere
    \]
\end{proof}

\subsection{Majorization}\label{sub.majority}

The author's attention was drawn to majorization theory by ChatGPT. This theory provides a systematic framework for comparing the unevenness of distributions of entries between two vectors. It underlies many classical inequalities and appears in various areas, such as convex analysis and optimization. The central theorem of this theory is the majorization inequality lemma, also named Karamata's inequality. We state a special case for the concave function $\log(x)$ in the following.

\begin{lem}[Karamata's inequality]\label{lem.karamata}
    Let $a=(a_1,\dots, a_m)$ and $b=(b_1,\dots, b_m)$ be two real vectors such that $a_1\geq \cdots \geq a_m>0$ and $b_1\geq \cdots \geq b_m>0$. If $a$ majorizes $b$, i.e., 
    \begin{enumerate}
        \item $a_1+\cdots +a_m=b_1+\cdots +b_m$,
        \item $a_1+\cdots +a_j\geq b_1+\cdots +b_j$ for $1\leq j<m$, 
    \end{enumerate}
    then $\log a_1+\cdots +\log a_m\leq \log b_1+\cdots +\log b_m$, equivalently, $a_1\cdots a_m\leq b_1\cdots b_m$.
\end{lem}

Before using it to prove Lemma \ref{lem.ubupton-1}, we give an alternative proof of Lemma \ref{lem.keygap} as a direct application of Karamata's inequality.

\begin{proof}[Proof of Lemma \ref{lem.keygap}]
    Consider $a=(\frac{1}{s_1},\dots,\frac{1}{s_{m-1}},\frac{1}{y_m}-t)$ and $x=(x_1,\dots, x_m)$. By definition, $\frac{1}{s_1}\geq \cdots\geq \frac{1}{s_{m-1}}\geq \frac{1}{y_m}-t>0$. As 
    \[
        \sum_{i=1}^ma_i=\frac{1}{s_1}+\cdots+\frac{1}{s_{m-1}}+\frac{1}{y_m}-t=1-t=\sum_{i=1}^mx_i,
    \]
    and $\frac{1}{s_1}+\cdots +\frac{1}{s_j}\geq x_1+\cdots +x_j$ for every $1\leq j< m$ by \cite{akn}*{Lemma 4.5}, we obtain that $a$ majorizes $x$. Then Karamata's inequality implies that 
    \[  
        x_1\cdots x_m\geq a_1\cdots a_m=\frac{1}{s_1}\dots\frac{1}{s_{m-1}}(\frac{1}{y_m}-t)=\frac{1}{y_m}(\frac{1}{y_m}-t)=\frac{1-ty_m}{y_m^2}. \qedhere
    \]
\end{proof}

\section{Proof of Theorem \ref{thm.main}}

This section is devoted to the proof of Theorem \ref{thm.main}. Let $X=\mathbb P(a_1,...,a_{n+1})$ be a well-formed weighted projective space of dimension $n\geq 4$ with canonical singularities, where $a_1\geq \cdots \geq a_{n+1}>0$. Let $q=a_1+\cdots +a_{n+1}$ be the Fano index of $X$, and set $\beta_i\coloneq \frac{a_i}{q}$ for each $i$. Then $\beta_1,\dots \beta_{n+1}$ are the barycentric coordinates of the interior point $0$ of the integral simplex associated with $X$. We begin by establishing several lemmas. 

\begin{lem}\label{lem.an}
    If $q>y_n(2y_n-1)$, then $a_{n}=y_n$.
\end{lem}

\begin{proof}
    By Theorem \ref{thm.lbofn-1}, we have $a_{n}=q\beta_{n}> \frac{y_n(2y_n-1)}{2y_n}=y_n-\frac{1}{2}$, and hence $a_{n}\geq y_n$. Suppose in contrast that $a_{n}\geq y_n+1=s_n$. Then 
    \begin{equation}\label{eq.lbofbetan-1}
        \beta_{n}\geq \frac{s_n}{q}.
    \end{equation} 
    By Theorem \ref{thm.ubton-1} and Lemma \ref{lem.lbupton-2}, we have 
    \begin{equation}\label{eq.1}
        \frac{\beta_{n}}{4y_{n-1}^2}< \beta_1\cdots \beta_{n}\leq \frac{1}{q},
    \end{equation}
    which implies that $\beta_{n}<\frac{4y_{n-1}^2}{q}<\frac{4y_{n-1}^2}{y_n(2y_n-1)}$. It is easy to check that $\frac{4y_{n-1}^2}{y_n(2y_n-1)}<\frac{1}{4y_{n-1}}$ by \eqref{eq.sylseq}.
    Hence we have
    \begin{equation}
        \beta_{n}<\frac{1}{4y_{n-1}},
    \end{equation}
    and $\beta_n+\beta_{n+1}\leq 2\beta_{n}<\frac{1}{y_{n-1}}$. Applying Lemma \ref{lem.keygap} to $m=n-1$, $x_i=\beta_{i}$ and $t=\beta_n+\beta_{n+1}\leq 2\beta_{n}<\frac{1}{y_{n-1}}$, we obtain that
    \begin{equation}\label{eq.2}
        \beta_1\cdots \beta_{n-1}\geq \frac{1-ty_{n-1}}{y_{n-1}^2}\geq \frac{1-2y_{n-1}\beta_{n}}{y_{n-1}^2}.
    \end{equation}
    Consider the function $f(x)\coloneq x (1-2y_{n-1}x )$, where $f'(x)=1-4y_{n-1}x>0$ for $0<x<\frac{1}{4y_{n-1}}$. That is, $f(x)$ is increasing in the range $(0,\frac{1}{4y_{n-1}})$. Combining \eqref{eq.lbofbetan-1}-\eqref{eq.2} yields that
    \[
        \frac{1}{q}\geq \beta_1\cdots \beta_{n}\geq \frac{\beta_{n}(1-2y_{n-1}\beta_{n})}{y_{n-1}^2}\geq\frac{\frac{s_n}{q}(1-2y_{n-1}\frac{s_n}{q})}{y_{n-1}^2}.
    \]
    As $q>y_n(2y_n-1)$, we have
    \[
        1\geq \frac{s_n}{y_{n-1}^2}\left(1-2y_{n-1}\frac{s_n}{q}\right)>\frac{s_n}{y_{n-1}^2}\left(1-\frac{2s_ny_{n-1}}{y_n(2y_n-1)}\right).
    \]
    On the other hand, substituting $s_n=y_n+1=y_{n-1}(y_{n-1}+1)+1$ into the right hand side of the above inequality, simple calculation shows that it is greater than $1$, which is a contradiction. Therefore, we must have $a_{n}=y_n$.
\end{proof}

\begin{lem}\label{lem.an+1}
    If $q>y_n(2y_n-1)$, then $a_{n+1}=y_n$.
\end{lem}

\begin{proof}
    By Lemma \ref{lem.an}, we have $a_{n+1}\leq a_{n}=y_n$. Suppose in contrast that $a_{n+1}\leq y_n-1$. Then
    \[
        t\coloneq \beta_n+\beta_{n+1}\leq \frac{2y_n-1}{q}<\frac{1}{y_n}<\frac{1}{y_{n-1}}.
    \]
    Applying Lemma \ref{lem.keygap} to $m=n-1$, $x_i=\beta_{i}$ and $t=\beta_n+\beta_{n+1}$, we obtain that
    \begin{equation}\label{eq.3}
        \beta_1\cdots \beta_{n-1}\geq \frac{1-ty_{n-1}}{y_{n-1}^2}\geq \frac{1-y_{n-1}\frac{2y_n-1}{q}}{y_{n-1}^2}.
    \end{equation}
    Combining with Theorem \ref{thm.ubton-1}, we have 
    \[
        \frac{1}{q}\geq \beta_1\cdots \beta_{n}=\beta_1\cdots \beta_{n-1}\frac{y_n}{q}\geq \left(\frac{1-y_{n-1}\frac{2y_n-1}{q}}{y_{n-1}^2}\right)\frac{y_n}{q}.
    \]
    As $q>y_n(2y_n-1)$, we have
    \[
        1\geq \frac{y_n}{y_{n-1}^2}\left(1-y_{n-1}\frac{2y_{n}-1}{q}\right)>\frac{y_n}{y_{n-1}^2}\left(1-\frac{y_{n-1}}{y_n}\right)=1,
    \]
    which is absurd. Therefore, we must have $a_{n+1}=y_n$.
\end{proof}

\begin{lem}\label{lem.ubupton-1}
    If $q>y_n(2y_n-1)$, then $a_{i}<\frac{q}{2y_n-1}(\frac{2y_n}{s_i})$ for $1\leq i\leq n-1$. 
\end{lem}

\begin{proof}
    It is equivalent to prove $\beta_i=\frac{a_i}{q}<\frac{2y_n}{s_i(2y_n-1)}=\frac{1}{s_i}+\frac{1}{s_i(2y_n-1)}$. Set $e_i\coloneq \frac{1}{s_i(2y_n-1)}$ and $r\coloneq \frac{2y_n}{q}-\frac{1}{y_n}$. By \eqref{eq.knownb} and assumption, we have 
    \[
        0\leq r<\frac{2y_n}{y_n(2y_n-1)}-\frac{1}{y_n}=\frac{1}{y_n(2y_n-1)}.
    \]
    By Lemmas \ref{lem.an} and \ref{lem.an+1}, we have $\beta_{n}=\beta_{n+1}=\frac{y_n}{q}$. Then Theorem \ref{thm.ubton-1} implies that $\beta_1\cdots \beta_{n-1}\frac{y_n}{q}=\beta_1\cdots \beta_{n}\leq \frac{1}{q}$, i.e.,
    \begin{equation}\label{eq.yleq1}
        y_n\beta_1\cdots \beta_{n-1}\leq 1.
    \end{equation}

    \begin{case}
        For $i=1$, we have $\beta_1\leq \frac{1}{2}$ by Theorem \ref{thm.psineq}, and our conclusion follows.
    \end{case}
    \begin{case}
        For $2\leq i\leq n-2$, suppose in contrast that $\beta_i\geq \frac{1}{s_i}+e_i$. Consider the vector
        \[
            z=\left(\frac{1}{s_1}, \frac{1}{s_2},\dots, \frac{1}{s_{i-1}}-e_i, \frac{1}{s_i}+e_i,\dots, \frac{1}{s_{n-2}}, \frac{1}{s_{n-1}}-r\right),
        \]
        which satisfies that
        \[
            \sum^{n-1}_{i=1}z_i=\sum^{n-1}_{i=1}\frac{1}{s_i}-\frac{2y_n}{q}+\frac{1}{y_n}=(1-\frac{1}{y_{n}})-\frac{2y_n}{q}+\frac{1}{y_n}=1-\frac{2y_n}{q}=\sum^{n-1}_{i=1}\beta_i.
        \]
        
        First, we claim that $z_1\geq \cdots\geq z_{n-1}>0$. It suffices to check the decreasing for $z_{i-1}=\frac{1}{s_{i-1}}-e_i$ and $z_i=\frac{1}{s_i}+e_i$, which comes from 
        \[
            z_{i-1}-z_i=\frac{s_i-s_{i-1}}{s_{i-1}s_i}-\frac{2}{s_i(2y_n-1)}\geq \frac{1}{s_{i-1}s_i}-\frac{2}{s_i(2y_n-1)}=\frac{2y_n-1-2s_{i-1}}{s_{i-1}s_i(2y_n-1)}>0.
        \]
        For $z_{n-1}>0$, we have $r<\frac{1}{y_n(2y_n-1)}<\frac{1}{s_{n-1}}$, and the claim follows.
    
        Second, we claim that $z$ majorizes $(\beta_1,\dots,\beta_{n-1})$. For $j\neq i-1$, we have 
        $\beta_1+\cdots+\beta_j\leq \frac{1}{s_1}+\cdots +\frac{1}{s_j}=z_1+\cdots +z_j$ by \cite{akn}*{Lemma 4.5}. For $j=i-1$, we have 
        \[
            \sum_{i=1}^{i-1}\beta_i=\sum_{i=1}^{i}\beta_i-\beta_i\leq \sum^{i}_{i=1}\frac{1}{s_i}-(\frac{1}{s_i}+e_i)=\sum_{i=1}^{i-1}z_i
        \]
        by \cite{akn}*{Lemma 4.5} and our assumption. By definition, $z$ majorizes $(\beta_1,\dots,\beta_{n-1})$. 
    
        Therefore, Karamata's inequality provides that $\beta_1\cdots \beta_{n-1} \geq z_1\cdots z_{n-1}$. Hence
        \[
            y_n\beta_1\cdots \beta_{n-1} \geq y_nz_1\cdots z_{n-1}=\left(1-\frac{s_{i-1}}{s_i(2y_n-1)}\right)\left(1+\frac{1}{2y_n-1}\right)\left(1-\frac{2y_ns_{n-1}}{q}+\frac{s_{n-1}}{y_n}\right). 
        \]
        Note that $\frac{s_{i-1}}{s_i}\leq \frac{2}{3}$ by definition and 
        \[
            \frac{2y_ns_{n-1}}{q}-\frac{s_{n-1}}{y_n}<\frac{s_{n-1}}{y_n(2y_n-1)}=\frac{1}{y_{n-1}(2y_n-1)}.
        \]
        Set $b\coloneq \frac{2}{3}$ and $c\coloneq \frac{1}{y_{n-1}}$. Then $1-b-c\geq \frac{1}{6}$ and $b+c-bc<1$, hence
        \begin{align*}
            y_n\beta_1\cdots \beta_{n-1} &>\left(1-\frac{b}{2y_n-1}\right)\left(1+\frac{1}{2y_n-1}\right)\left(1-\frac{c}{2y_n-1}\right)\\
            &=1+\frac{1-b-c}{2y_n-1}-\frac{b+c-bc}{(2y_n-1)^2}+\frac{bc}{(2y_n-1)^3}\\
            &>1+\frac{1}{6(2y_n-1)}-\frac{1}{(2y_n-1)^2}=1+\frac{2y_n-7}{6(2y_n-1)^2}>1,
        \end{align*}
        which contradicts \eqref{eq.yleq1}.
    \end{case}
    
    \begin{case}
        For $i=n-1$, suppose in contrast that $\beta_{n-1}\geq \frac{1}{s_{n-1}}+e_{n-1}$. Consider the vector
    \[
        z=\left(\frac{1}{s_1}, \frac{1}{s_2},\dots, \frac{1}{s_{n-3}},\frac{1}{s_{n-2}}-r-e_{n-1}, \frac{1}{s_{n-1}}+e_{n-1}\right),
    \]
    which satisfies as above that
    \[
        \sum^{n-1}_{i=1}z_i=\sum^{n-1}_{i=1}\frac{1}{s_i}-r=1-\frac{2y_n}{q}=\sum^{n-1}_{i=1}\beta_i.
    \]
    
    To show that $z_1\geq \cdots\geq z_{n-1}>0$, it suffices to prove $z_{n-2}\geq z_{n-1}$. Note that 
    \[
        r+2e_{n-1}<\frac{1}{y_n(2y_n-1)}+2e_{n-1}<3e_{n-1}=\frac{3}{s_{n-1}(2y_n-1)}<\frac{1}{s_{n-2}s_{n-1}}\leq \frac{1}{s_{n-2}}-\frac{1}{s_{n-1}},
    \]
    and our conclusion follows.

    To show that $z$ majorizes $(\beta_1,\dots,\beta_{n-1})$, it suffices to check the majorization for $j=n-2$  by \cite{akn}*{Lemma 4.5}. In this case, we have
    \[
        \sum_{i=1}^{n-2}\beta_i=\sum_{i=1}^{n-1}\beta_i-\beta_{n-1}\leq (\sum^{n-1}_{i=1}\frac{1}{s_i}-r)-(\frac{1}{s_{n-1}}+e_{n-1})=\sum_{i=1}^{n-2}z_i,
    \]
    which is the desired result. 

    Therefore, Karamata's inequality provides that $\beta_1\cdots \beta_{n-1} \geq z_1\cdots z_{n-1}$ and 
    \[
        y_n\beta_1\cdots \beta_{n-1} \geq y_nz_1\cdots z_{n-1}=\left(1-(r+e_{n-1})s_{n-2}\right)\left(1+e_{n-1}s_{n-1}\right). 
    \]
    Note that $(r+e_{n-1})s_{n-2}<\left(\frac{1}{y_n(2y_n-1)}+\frac{1}{s_{n-1}(2y_n-1)}\right)s_{n-2}=\frac{1}{y_{n-2}(2y_n-1)} \leq \frac{1}{2(2y_n-1)}$. 
    Hence 
    \begin{align*}
        y_n\beta_1\cdots \beta_{n-1} &>\left(1-\frac{1}{2(2y_n-1)}\right)\left(1+\frac{1}{2y_n-1}\right)\\
        &=1+\frac{1}{2(2y_n-1)}-\frac{1}{2(2y_n-1)^2}\\
        &=1+\frac{2y_n-2}{2(2y_n-1)^2}>1,
    \end{align*}
    which contradicts \eqref{eq.yleq1}. \qedhere
    \end{case}
    \setcounter{case}{0}
\end{proof}

Now we are ready to prove the main theorem.

\begin{proof}[Proof of Theorem \ref{thm.main}]
    Suppose in contrast that $q>y_n(2y_n-1)$. Then Lemmas \ref{lem.an}-\ref{lem.ubupton-1} provide that $\frac{(2y_n-1)a_i}{q}<1$ for $i\geq n$ and $\frac{(2y_n-1)a_i}{q}<\frac{2y_n}{s_i}$ for $1\leq i\leq n-1$, where $\frac{2y_n}{s_i}$ are integers. Hence
    \begin{align*}
        \age(2y_n-1)&=2y_n-1-\sum_{i=1}^{n+1}\lfloor\frac{(2y_n-1)a_i}{q}\rfloor\\
        &\geq 2y_n-1-\sum_{i=1}^{n-1}(\frac{2y_n}{s_i}-1)\\
        &=2y_n-1-2y_n(1-\frac{1}{y_n})+(n-1)=n.
    \end{align*}
    This contradicts Proposition \ref{prop.criterion}.
\end{proof}

\section{Distribution of Fano indices}\label{sec.4}

This section is devoted to the study of the distribution of Fano indices.

\begin{defn}
    We define $\mathrm{I}_{n,\wps}$ to be the set of Fano indices of $n$-dimensional well-formed weighted projective spaces with canonical singularities, together with $\{1,\dots,n\}$; define $\mathrm{I}_{n,\term}$ to be the set of indices of $n$-dimensional terminal Calabi--Yau varieties; and define $\mathrm{I}_{n,\loc}$ to be the set of indices of an $n$-dimensional log canonical singularity $P\in X$.
\end{defn}

It is already observed in \cite{liu2025}*{\S 1.1} that the above three sets coincide in low dimensions. We therefore propose the following conjecture that they coincide in all dimensions.

\begin{conj}\label{conj.coincide}
    $\mathrm{I}_{n,\wps}=\mathrm{I}_{n,\term}=\mathrm{I}_{n,\loc}$ in any dimension $n\geq 2$.
\end{conj}

\begin{prop}\label{prop.leq3}
    Conjecture \ref{conj.coincide} holds in dimension $n\leq 3$.
\end{prop}

\begin{proof}
    For $n=2$, the well-formed weighted projective spaces are $\mathbb P^2, \mathbb P(2,1,1), \mathbb P(3,2,1)$, with Fano indices $3,4$ and $6$ respectively. Combining this with the well-known facts (cf. \cite{machida-oguiso}, \cite{fujino}), the common index set is 
    \[
        \{m\in\mathbb Z_{>0}\mid \varphi(m)\leq 2\}=\{1,2,3,4,6\},
    \]
    where $\varphi(m)$ is the Euler function of $m$.
    For $n=3$, it follows from \cite{kasprzyk}*{Table 3},  \cite{masamura}*{Proposition 3.2} and \cite{fujino}*{Theorem 0.1} that the common index set is 
    \[
        \{m\in\mathbb Z_{>0}\mid \varphi(m)\leq 20\} \setminus \{60\}. \qedhere
    \]
\end{proof}

From now on, we study Fano indices of $4$-dimensional well-formed weighted projective spaces with canonical singularities. The upper bound $q=y_4(2y_4-1)=3486$ is also obtained in \cite{kasprzyk}*{Theorem 3.6}. The sets $A_i$ in the following theorem are disjoint and the subscript $i$ in $A_i$ refers to the initial term $y_i$ for $1\leq i\leq 4$.

\begin{thm}\label{thm.disindim4}
    We have $\mathrm{I}_{4,\wps} \subseteq A_1\cup A_2\cup A_6\cup A_{42}$, where
    \begin{align*}
        A_1\coloneq &\{m\in \mathbb Z \mid1\leq m\leq 1743\},\\
        A_2\coloneq &\{m\in2\mathbb Z \mid 1743< m\leq 2324\},\\
        A_6\coloneq &\{m\in6\mathbb Z \mid 2324< m\leq 2988\},\\
        A_{42}\coloneq &\{m\in 42\mathbb Z \mid 2988< m\leq 3486\}.
    \end{align*}
    Moreover, $\mathrm{I}_{4,\wps} \subseteq \{m\in\mathbb Z_{>0}\mid \varphi(m)\leq 984\}$.
\end{thm}
\begin{proof}
    One can verify this using \cite{bk}, or by running a computer search for the Fano index $q=a_1+\cdots+a_5$ of $\mathbb P(a_1,\dots, a_5)$. To make the search more efficient, we first establish sharper bounds for $a_i$ by Corollary \ref{cor.bounds}. 
    \begin{claim}
    Suppose $q\geq 9$. 
    \begin{enumerate}
        \item If $a_1=\frac{q}{2}$, then $\frac{q}{7}\leq a_2\leq \frac{q}{3}$ and $\frac{q-a_1-a_2}{3}\leq a_3\leq \frac{q}{5}$.
        \item If $\frac{2q}{5}<a_1<\frac{q}{2}$, then $\frac{q}{6}\leq a_2\leq \frac{q}{3}$ and $\frac{q-a_1-a_2}{3}\leq a_3\leq \frac{q}{5}$.
        \item If $a_1=\frac{2q}{5}$, then  $\frac{q}{6}\leq a_2\leq \frac{q}{3}$ and $\frac{q-a_1-a_2}{3}\leq a_3\leq \frac{q}{4}$.
        \item If $\frac{q}{4}\leq a_1<\frac{2q}{5}$, then $\frac{q}{5}\leq a_2\leq \frac{q}{3}$.
    \end{enumerate}
    \end{claim}

    \begin{proof}
        It follows from Corollary \ref{cor.bounds} that $\frac{q}{4}\leq a_1\leq \frac{q}{2}$, $a_i\leq \frac{q}{1+i}$ and if $a_1<\frac{2q}{3+i}$, then $a_i\geq \frac{q}{3+i}$ for $1<i\leq 5$.
    
        If $a_1> \frac{2q}{5}$, then $a_3\leq \frac{q}{5}$. Otherwise, $q<5a_3\leq 5a_2\leq \frac{5q}{3}<2q$ and $a_5\leq a_4\leq \frac{q}{5}$. Hence 
        \[
            1\leq \age(5)\leq 5-2-2-\lfloor\frac{5a_4}{q}\rfloor=1-\lfloor\frac{5a_4}{q}\rfloor\leq 1.
        \]
        It follows that $\age(5)=1$ and $\frac{5a_4}{q}\neq 1$, and hence $\age (q-5)=4$ by Proposition \ref{prop.symmetry}. This contradicts Proposition \ref{prop.criterion}.
    
        In case (1), if $a_2<\frac{q}{7}$, then $\age(7)=7-\lfloor\frac{7a_1}{q}\rfloor=4$, contradicting Proposition \ref{prop.criterion}. 
        
        In cases (2) and (3), if $a_2<\frac{q}{6}$, then $\age(6)=6-\lfloor\frac{6a_1}{q}\rfloor=4$, a contradiction. 
    
        In case (4),  if $a_2<\frac{q}{5}$, then $\age(5)=5-\lfloor\frac{5a_1}{q}\rfloor=4$, a contradiction.
    \end{proof}

    Our goal is then to search for $q\leq 3528$ (see \eqref{eq.knownb}) whose partition $(a_1, a_2, a_3, a_4,a_5)$ meets the following conditions:
    \begin{enumerate}
        \item $a_1\geq a_2 \geq a_3 \geq a_4 \geq a_5>0$ lie within the ranges given in the above Claim;
        \item the tuple is well-formed, i.e., $\gcd(a_1,\dots, \hat a_i,\dots, a_{n+1})=1$ for every $i$;
        \item $\age(k)\neq 4$ for any $2\leq k\leq q-2$.
    \end{enumerate}
    
    It took approximately 155 hours to list all possible $q$ on a personal computer; see https://doi.org/10.5281/zenodo.21806548 for the output list and the code. Now the conclusions are obtained by directly checking the resulting list.
\end{proof}

\begin{rem}
    Note that $\varphi(y_4(2y_4-1))=\varphi(3486)=984$, yet $3486$ is not the largest number $m$ such that $\varphi(m)\leq 984$. In general, we can ask the following question:
    \begin{ques}
        Do we have $\varphi(m)\leq \varphi(y_n(2y_n-1))$ for any $m\in \mathrm{I}_{n,\wps}$?
    \end{ques}
\end{rem}

\begin{rem}
Suppose 
\[
    \mathrm{J}_{4}\coloneq \{m\in\mathbb Z_{>0}\mid  m\leq 3486,~ \varphi(m)\leq 984\}\backslash \mathrm{I}_{4,\wps}.
\]
This set can be explicitly described by the resulting list of Theorem \ref{thm.disindim4}; for instance, $J_4\cap A_{42}= \{3444=y_4(2y_4-2)\}$. For the remaining parts, we can tell that $\# \{J_4\cap A_{6}\}=26$,  $\# \{J_4\cap A_{2}\}=84$, and  $\# \{J_4\cap A_{1}\}=157$.
\end{rem}

Suggested by Theorem \ref{thm.disindim4}, we propose the following conjecture. 

\begin{conj}\label{conj.condis}
    Let $X=\mathbb P(a_1,...,a_{n+1})$ be a well-formed weighted projective space of dimension $n\geq 4$, where $q=a_1+\cdots +a_{n+1}$ is the Fano index of $X$. 
    \begin{enumerate}
        \item If $q>\frac{y_i}{s_i}y_n(2y_n-1)$ for $1\leq i<n$, then $s_i\mid q$;
        \item If $q>y_{n-1}^2(2y_n-1)$, then $y_n\mid q$.
    \end{enumerate}
\end{conj}

\begin{rem}
    Note that Conjecture \ref{conj.condis}(1) does not hold in dimension $3$, as shown by \cite{kasprzyk}*{Table 3}. Note also that Conjecture \ref{conj.condis}(2) is a consequence of Conjecture \ref{conj.condis}(1), as if $q>y_{n-1}^2(2y_n-1)=\frac{y_{n-1}}{s_{n-1}}y_n(2y_n-1)$, then $s_i\mid q$ for all $1\leq i<n$, where $s_i$ are pairwise coprime; hence $y_n=s_1\cdots s_{n-1}\mid q$. As Theorem \ref{thm.disindim4} shows, Conjecture \ref{conj.condis}(1)(2) hold in dimension $4$.
\end{rem}

\begin{rem}
    Conjecture \ref{conj.condis} predicts that, for large $q$, there exists a gap $y_n$ between consecutive possible Fano indices. Hence assuming Conjecture \ref{conj.condis}(2) holds, we can regain that $q\leq y_n(2y_n-1)$.
    We can further conjecture the following two special cases:
    \begin{enumerate}
        \item if $q=y_n(2y_n-1)$ then $X\cong \mathbb P (\frac{q}{s_1},\frac{q}{s_2}, \dots,\frac{q}{s_{n-2}}, y_{n-1}, y_{n-1}-1)$ in \cite{wang}*{Theorem 3.2};
        \item $q\neq y_n(2y_n-2)$.
    \end{enumerate} 
    Both cases are expected to follow by the same argument as in Theorem \ref{thm.main}. 
\end{rem}

Assuming Conjecture \ref{conj.coincide} holds, we would also gain some insight into the distribution of terminal Calabi--Yau varieties.

\begin{prop}
    Suppose Conjecture \ref{conj.coincide} holds in dimension $4$. Then the index of a terminal Calabi--Yau $4$-fold is at most $3486$.  
\end{prop}

\begin{rem}
    The largest known third Betti number $b_3$ of a smooth Calabi--Yau $3$-fold is $984=\varphi(3486)$, due to \cite{kaeuzer-skarke}. It arises from the complete classification of $4$-dimensional reflexive polyhedra (equivalently, Gorenstein toric Fano 4-folds): taking general members of the anti-canonical linear systems yields Gorenstein Calabi--Yau 3-folds, whose resolutions give smooth Calabi--Yau 3-folds, one of which attained the largest known third Betti number. Combined with the lower dimension cases, we would like to conjecture that 
    \begin{equation}
        b_n\leq \varphi(y_{n+1}(2y_{n+1}-1))+1+(-1)^n
    \end{equation}
    for all Calabi--Yau $n$-folds. Note that even if this conjecture turns out to be true, it does not mean the boundedness of Calabi--Yau $n$-folds.
\end{rem}

\end{document}